\documentclass{article}
\usepackage[utf8]{inputenc}
\usepackage{amsmath}
\usepackage{amsthm}
\usepackage{amssymb}
\usepackage{tikz}
\usetikzlibrary{arrows, automata}
\usetikzlibrary{intersections}
\usepackage{hyperref}
\usepackage{cite}
\usepackage{authblk}
\usepackage{graphicx}
\usepackage{adjustbox}
\usepackage{caption}
\usepackage{subcaption}
\usepackage{appendix}
\usepackage{pgfplots}
\usepgfplotslibrary{fillbetween}
\pgfplotsset{compat=1.11}
\pgfdeclarelayer{bg}
\pgfsetlayers{bg,main}

\newtheorem{theorem}{Theorem}
\newtheorem{lemma}{Lemma}
\newtheorem{corollary}{Corollary}

\DeclareMathOperator*{\argmin}{argmin}

\providecommand{\subjclass}[1]{\textbf{2010 AMS Subject Class:} #1}
\providecommand{\keywords}[1]{\textbf{Keywords:} #1}
\title{Independent Domination in Directed Graphs}
\author[1\footnote{Email: macary@mix.wvu.edu}]{Michael Cary}
\author[2]{Jonathan Cary}
\author[3\footnote{Email: drsavariprabhu@gmail.com}]{Savari Prabhu}
\affil[1]{Division of Resource Economics and Management, West Virginia University}
\affil[2]{Department of Mathematics, Virginia Commonwealth University}
\affil[3]{Department of Mathematics, Sri Venkateswara College of Engineering, Sriperumbudur 602117, India}
\begin{document}
\maketitle
\begin{abstract}
In this paper we initialize the study of independent domination in directed graphs. We show that an independent dominating set of an orientation of a graph is also an independent dominating set of the underlying graph, but that the converse is not true in general. We then prove existence and uniqueness theorems for several classes of digraphs including orientations of complete graphs, paths, trees, DAGs, cycles, and bipartite graphs. We also provide the idomatic number for special cases of some of these families of digraphs.
\end{abstract}



\section{Introduction}
The dominating set problem has its basis in finding efficient communication points for transmitting information throughout a network \cite{dunbar2006broadcasts}. Finding a minimum set of vertices which dominates the rest of the graph corresponds to finding a smallest possible source comprised of members of the network from which to spread information throughout the rest of the network in the most efficient manner possible \cite{wu1999calculating}. It cannot be understated that, while minimum may and typically does refer to the cardinality of the dominating set, it could also refer to a minimum dominating set with respect to any other metric, e.g., a minimum dominating set with respect to the cost of obtaining and/or maintaining the dominating set, or minimum in the sense of minimizing the cost of broadcasting information given some cost functional. However, in the context of this paper, any question about minimum dominating sets will be with respect to the cardinality.

To reflect the diversity of real world communication (or other) network requirements, different sets of constraints are often imposed on dominating sets. One such instance occurs when all members of the dominating set need to be connected, such as in backbones of ad hoc (non-time-invariant) networks \cite{dai2006constructing}. Another instance occurs when the dominating set needs to be independent, such as is the case when transmitting information that is at risk of fading \cite{wu2006extended}. It is this version of the dominating set problem, that of independent dominating sets, that will be the focus of this paper. Particularly, as a network may admit only directional flows (i.e., information is transmitted in only one direction), we will consider the case of independent dominating sets in directed graphs.

Formally, a dominating set is a subset $D\subseteq V(G)$ of the vertex set $V(G)$ of a graph $G$ satisfying $D\cup N(D) = V$. A minimum dominating set is a solution $\hat{D}\in\argmin\limits_{D\in\mathcal{D}(G)} |D|$ where $\mathcal{D}(G)$ denotes the set of all dominating sets of $G$. An independent dominating set is a dominating set that is also an independent set. Work on the undirected version of this topic can be traced back to \cite{cockayne1976disjoint}. More recent graph theoretic results include showing that the problem of finding a minimum independent dominating set is, in general, NP-Hard \cite{irving1991approximating}, a result which was recently extended to independent rainbow domination \cite{shao2019independent}. Several Nordhaus-Gaddum type results on independent domination were established in \cite{cockayne1995nordhaus}. Additional research on bounding the independent domination numbers of graphs, i.e., the size of a smallest independent dominating set, is vast. For example, \cite{favaron1990bound} improved bounds on the independent domination number of trees, \cite{glebov1998independent} established bounds for graphs with given minimum degree, \cite{ma2004note} proved results on independent dominating sets in bipartite graphs, and \cite{sun1999upper} established general upper bounds for the independent domination number of graphs. Other results on independent dominating sets include results on random cubic graphs \cite{duckworth2002minimum} and random regular graphs \cite{duckworth2006independent}. The approaches developed to find independent dominating sets in graphs have led not only to progress in other areas of domination such as dominator colorings of graphs \cite{gera2007dominator} and digraphs \cite{cary2019dominator}, but have also led to applications in decycling graphs \cite{bau2002decycling} as well as in mathematical chemistry \cite{prabhu2018independent}.

In order to discuss dominating and independent dominating sets in digraphs, a few important distinctions pertaining to notation need to be established. First, all (di)graphs considered are both simple and finite. The term digraph will be used as a general term to discuss orientations of graphs. This clarification is vital, as the adjective ``directed" will refer to the literal directed case of a graph, e.g., a directed path $P_{n}=v_{1}v_{2}\dots v_{n}$ which has the arc set $A(P_{n})=\{v_{i}v_{i+1}\ |\ 1\leq i<n\}$. Since we are discussing digraphs, $D$ will be reserved for referencing digraphs and $G$ will refer to the underlying (undirected) graph of the digraph $D$. The standard notation for the vertex set will be used ($V(D)=V(G)$), and the edge set $E(G)$ becomes the arc set $A(D)$. Independent dominating sets (and dominating sets) will be referred to using the notation $ID$ with superscript signifiers used to distinguish specific (independent) dominating sets. For any vertex $v\in V(D)$, the in-degree and out-degree of $v$ will be represented by $d^{-}(v)$ and $d^{+}(v)$, respectively. The open in-neighborhood of $v$, the set of vertices with arcs leading into $v$, will be denoted by $N^{-}(v)=\{u\in V(D)\ |\ uv\in A(D)\}$. The open out-neighborhood of $v$ is similarly defined as $N^{+}(v)=\{u\in V(D)\ |\ vu\in A(D)\}$. The closed in- and out-neighborhoods of $v$, the union of the open in- and out-neighborhoods of $v$ with $v$ itself, are denoted by $N^{-}[v] = N^{-}(v)\cup\{v\}$ and $N^{+}[v] = N^{+}(v)\cup\{v\}$, respectively. The symmetric difference of two sets $A$ and $B$ is denoted $A\Delta B$ and represents the set $(A\cup B)\setminus (A\cap B)$, i.e., the union minus the intersection. The reversal of a digraph $D$, denoted $D^{-}$, has the same vertex set as $D$, but has the direction of each arc reversed from its orientation in $A(D)$. The domatic number, defined as the maximum number of pairwise disjoint dominating sets \cite{farber1984domination}, is typically denoted by $d(G)$ for some graph $G$. Since we are studying independent domination in digraphs, we study a variant called the idomatic number, defined as the maximum number of pairwise disjoint independent dominating sets, denoted by $id(D)$ for some digraph $D$. Idomatic numbers have been studied in undirected graphs \cite{valencia2010idomatic} and \cite{zelinka1983adomatic}, but remain a novel topic in digraphs. Finally, a trivial (di)graph is a (di)graph on one vertex. Any additional uses of notation may be assumed to come from the standard reference texts on domination \cite{haynes1998} and \cite{haynes2017domination}. For a reference paper specifically dedicated to results on independent dominating sets in graphs, the reader is referred to \cite{goddard2013independent}.

Throughout this paper, all (di)graphs are assumed to be simple, connected, and finite. The remainder of this paper will begin by asking and answering a natural question relating independent dominating sets in digraphs to independent dominating sets in the underlying graph of the digraph. To help illustrate this initial result, we first study the existence and uniqueness of independent dominating sets in tournaments (orientations of complete graphs). We then study the existence and uniqueness of independent dominating sets in orientations of paths, trees, DAGs (directed acyclic graphs), cycles, and bipartite graphs. Once these results are in place, we build on them by providing some initial results on the idomatic number of special cases of some of these families of digraphs. The paper will then conclude with a provision of possible avenues for furthering this line of research.

\section{Independent Dominating Sets in Digraphs}
We begin our study by asking a fundamental question about how orienting a graph might affect independent dominating sets. If orientation was not relevant, then the study of independent dominating sets and idomatic number of digraphs would be irrelevant. It may appear obvious (and the following theorem shows that it is true) that an independent dominating set of a digraph is also an independent dominating set of its underlying graph. It turns out, however, that orientation is crucial to determining whether or not a given independent dominating set of the underlying graph of a digraph is an independent dominating set in any particular orientation thereof. Given a graph $G$ and an independent dominating set $ID$ of $G$, not every orientation of $G$ will preserve the property that $ID$ remains an independent dominating set.

\begin{theorem}
The following statements about independent dominating sets are all true:
\begin{itemize}
\item [1:] {Every independent dominating set $ID$ of a digraph $D$ in an independent dominating set of the underlying (undirected) graph $G$.}
\item [2:] {For every graph $G$ there exists some orientation $D$ such that, for any particular independent dominating set $ID$ of $G$, $ID$ is an independent dominating set of $D$.}
\item [3:] {For every non-trivial graph $G$, there exists some orientation $D$ such that, for any particular independent dominating set $ID$ of $G$, $ID$ is not an independent dominating set of $D$.}
\end{itemize}
\end{theorem}
\begin{proof}
1: $ID$ dominates $G$ since any arc $xy\in A(D)$ is also an edge $xy\in E(G)$. $ID$ is independent in $G$ since any edge $xy\in E(G)$ that would cause $ID$ to not be independent in $G$ exists as either the arc $xy$ or the arc $yx$ in $A(D)$ which would contradict that $ID$ is an independent set in $D$.

2: Let $ID$ be an independent dominating set of some graph $G$. Orient all edges incident with vertices in $ID$ away from $ID$, towards $V(G)\setminus ID$.

3: Since $G$ is non-trivial, every independent dominating set must be a proper subset of the vertex set, else it is not independent. let $ID$ be the independent set of $G$ in question. Let $D$ be an orientation of $G$ formed by orienting all arcs from $V(G)\setminus ID$ to $ID$. The set $ID$ does not dominate $V(G)\setminus ID$ in $D$ and is therefore not an independent dominating set in the orientation $D$.
\end{proof}

To help illustrate this result, the following existence and uniqueness theorem for independent dominating sets of tournaments provides an easy opportunity to visualize each of the above claims since any independent dominating set of $K_{n}$ is a single vertex and every vertex represents an independent dominating set of $K_{n}$.

\begin{theorem}\label{tourns}
Let $T_{n}$ be a tournament. $T_{n}$ has an independent dominating set $ID\ \iff\ \exists\ v\in V(T_{n})$ such that $vu\in A(T_{n})\ \forall\ u\in V(T_{n})$. Furthermore, any independent dominating set of a tournament must be unique.
\end{theorem}
\begin{proof}
($\impliedby$) If $vu\in A(T_{n})\ \forall\ u\in V(T_{n})\setminus\{v\}$, then we may take $ID=\{v\}$ as our independent dominating set of $T_{n}$.

($\implies$) Since $T_{n}$ is a tournament, every maximal independent set consists of a single vertex. Since $T_{n}$ admits an independent dominating set $ID$, it follows that $ID=\{v\}$ for some $v\in V(T_{n})$ which in turn implies that $vu\in A(T_{n})\ \forall\ u\in V(T_{n})\setminus\{v\}$.

(Uniqueness) Since an independent dominating set of a tournament consists of a single vertex $v$ which dominates all other vertices, no other vertex dominates $v$, hence $ID=\{v\}$ is the unique independent dominating set.
\end{proof}

The remainder of this section is dedicated to establishing existence and uniqueness theorems for several other families of digraphs. Each family considered will be given its own subsection. We begin this analysis with orientations of paths.

\subsection{Paths}
Recall our use of the term ``directed path" to refer to the path $P_{n}=v_{1}v_{2}\dots v_{n}$ which has the arc set $\{v_{i}v_{i+1}\ |\ 1\leq i<n\}$, and that we use the phrase orientations of paths as a means to address all orientations of one or more paths.

\begin{figure}[h!]
\centering
\label{directedpathfigure}
\begin{tikzpicture}[-,>=stealth',shorten >=1pt,auto,node distance=2cm,
                    thick,main node/.style={circle,draw}]
  \node[main node] (A)					      {$v_{1}$};
  \node[main node] (B) [right of=A]		      {$v_{2}$};
  \node[main node] (C) [right of=B] 	      {$v_{3}$};
  \node[main node] (D) [right of=C]      	  {$v_{4}$};
  \node[main node] (E) [right of=D]           {$v_{5}$};
  \draw[thick,->]  (A) to [bend left]         (B);
  \draw[thick,->]  (B) to [bend right]        (C);
  \draw[thick,->]  (C) to [bend left]         (D);
  \draw[thick,->]  (D) to [bend right]        (E);
\end{tikzpicture}
\caption{An example of the directed path of length five.}
\end{figure}
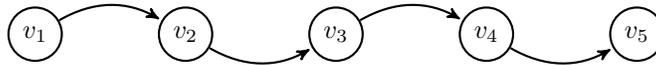

We begin by providing an existence and uniqueness result specifically for the directed path.

\begin{lemma}\label{dipath}
Let $P_{n}=v_{1}v_{2}\dots v_{n}$ be a directed path. Then the set $ID=\{v_{i}\ |\ i\equiv 1\ (\mathrm{mod}\ 2)\}$ is the unique independent dominating set of $P_{n}$.
\end{lemma}
\begin{proof}
By definition of $P_{n}=v_{1}v_{2}\dots v_{n}$ and since $v_{i}$ dominates $v_{i+1}$, the set $ID=\{v_{i}\ |\ i\equiv 1\ (\mathrm{mod}\ 2)\}$ is an independent dominating set of $P_{n}$. To show that $ID$ is unique, it suffices to observe that the vertex $v_{1}$ must be in every independent dominating set of $P_{n}$ and to show that for every independent dominating set $ID$ of $P_{n}$, there can be no two consecutive vertices of $P_{n}$ that do not belong to $ID$. This second claim holds since, if neither $v_{i}$ nor $v_{i+1}$ belong to a particular independent dominating set $ID$, then the vertex $v_{i+1}$ is neither dominated by any member of $ID$, hence $ID$ is not actually a dominating set of $P_{n}$.
\end{proof}

To determine if every oriented path admits at least one independent dominating set, it would be useful to know if any vertices have to be a member of \emph{every} independent dominating set of an orientation of a path. The following lemma proves that vertices of a path which have in-degree zero necessarily must be a part of every independent dominating set of that oriented path. In fact, we can generalize this claim to cover all digraphs.

\begin{lemma}\label{indegzero}
Let $D$ be a digraph. Every vertex $v\in V(D)$ with $d^{-}(v)=0$ must be in every independent dominating set of $D$.
\end{lemma}
\begin{proof}
Since no vertex dominates $v$, it immediately follows that $v$ must be in every independent dominating set of $D$.
\end{proof}

With these initial results in place, we can now prove an existence theorem which demonstrates that all oriented paths admit an independent dominating set.

\begin{theorem}\label{pathstheorem}
Every oriented path admits an independent dominating set.
\end{theorem}
\begin{proof}
The proof is by induction with basis $P_{1}$. Let $P_{n}=v_{1}v_{2}\dots v_{n}$ and let $P_{n-1}=v_{1}v_{2}\dots v_{n-1}$ be the subpath of $P$ comprised of the first $n-1$ vertices and first $n-2$ arcs of $P_{n}$. Since $P_{n-1}$ is smaller than $P_{n}$, it follows that $P_{n-1}$ admits an independent dominating set, call it $ID^{\prime}$.

First consider the case where $v_{n-1}v_{n}\in A(P_{n})$. If $v_{n-1}\in ID^{\prime}$ then $ID^{\prime}\cup N(ID^{\prime}) = V(P_{n})$ and if $v_{n-1}\not\in ID^{\prime}$, then we may take $ID = ID^{\prime}\cup\{v_{n}\}$ and obtain an independent dominating set of $P$. In either case, we have established an independent dominating set of $P_{n}$.

Next, consider the case where $v_{n}v_{n-1}\in A(P_{n})$. If $v_{n-1}\not\in ID^{\prime}$, then we may take $ID=ID^{\prime}\cup\{v_{n}\}$ as our independent dominating set of $P_{n}$, so we may assume that $v_{n-1}\in ID^{\prime}$ for every independent dominating set $ID^{\prime}$ of $P_{n-1}$. Now if $v_{n-2}v_{n-1}\in A(P_{n-1})$, then it must be the case that $v_{n-3}v_{n-2}\in A(P_{n-1})$, else $v_{n-2}$ is neither dominated by a vertex in $ID^{\prime}$ nor in $ID^{\prime}$ itself, contradicting the fact that $ID^{\prime}$ is an independent dominating set of $P_{n-1}$. Therefore we may take $ID=[ID^{\prime}\setminus\{v_{n-1}\}]\cup\{v_{n}\}$ as our independent dominating set of $P_{n}$. This leaves only the case in which $v_{n-1}v_{n-2}\in A(P_{n-1})$.

Beginning with $v_{n}$ and proceeding backwards through $P_{n}$, every vertex must have out-degree one until we reach some vertex $v_{i}$ which has out-degree zero. Per Lemma \ref{dipath}, the directed path from $v_{n}$ to $v_{i}$ only admits dominating sets consisting of vertices with indices of the same parity (mod 2). If the vertex $v_{i-1}$ is a member of $ID^{\prime}$, then we can consider the directed path from $v_{n}$ to $v_{i+1}$, call it $P^{\star}$ and replace $P^{\star}\cap ID^{\prime}$ with the symmetric difference $P^{\star} \Delta [P^{\star}\cap ID^{\prime}]$. By denoting $S=P^{\star}\cap ID^{\prime}$ and $S^{\star}=P^{\star} \Delta [P^{\star}\cap ID^{\prime}]$, we arrive at the independent dominating set $ID=[ID^{\prime}\setminus s]\cup S^{\star}$ of $P_{n}$. Finally, if $v_{i-1}$ is not a member of $ID^{\prime}$, we may set $P^{\star}=v_{n}\dots v_{i}$ and use the same argument as before where $ID=[ID^{\prime}\setminus s]\cup S^{\star}$ (the difference is that $v_{i}$ is also a member of $ID$ in this case). With all cases covered, the inductive step is complete and we conclude that every oriented path admits an independent dominating set.
\end{proof}

\subsection{Trees}
We next turn our attention to orientations of trees by proving an existence theorem for all orientations of trees.

\begin{theorem}
Every oriented tree admits an independent dominating set,
\end{theorem}
\begin{proof}
The proof is by contradiction of a minimum counterexample. Let $T$ be a minimum orientation of a tree, with respect to the cardinality of its vertex set, that does not admit an independent dominating set. Since $T$ is an orientation of a tree, it follows that $T$ has at least one vertex with in-degree zero, call it $v$. Create the subgraph $T^{\prime}$ of $T$ as the induced subgraph $T^{\prime}=T[V(T)\setminus N^{+}[v]]$. Since $T^{\prime}$ is smaller than $T$, it follows that $T$ admits an independent dominating set (note that $T^{\prime}$ may actually be a forest). Choose any independent dominating set of $T^{\prime}$, call it $ID$. We can then create an independent set $ID^{\star}=ID\cup\{v\}$ of $T$, contradicting that $T$ is a smallest counterexample to the claim, thereby concluding that no smallest counterexample exists and completing the proof.
\end{proof}

Knowing that all oriented trees admit an independent dominating set, we now show that arborescences, or out-trees, as well as anti-arborescences, or in-trees, have a unique independent dominating set. To do so, we will first present a quick lemma about independent dominating sets and reversal of orientation, similar to the spirit of the main result in \cite{cary2019trees}.

\begin{lemma}\label{reversereverse}
Let $D$ be a digraph which admits at least one independent dominating set, let $D^{-}$ be its reversal, and let $ID$ be an independent dominating set of $D$. Then the set $V(D)\setminus ID$ is a dominating set in $D^{-}$.
\end{lemma}
\begin{proof}
Note that $V(D)\setminus ID$ might not be an independent set. It suffices to show that $\forall v\in ID$, $\exists u\in V(D)\setminus ID$ such that $uv\in A(D^{-}$. This follows directly from the definition of $ID$ being an independent dominating set, in particular that $\forall u\in V(D)\setminus ID$, $\exists v\in ID$ such that $vu\in A(D)$.
\end{proof}

We now show that all arborescences and anti-arborescences admit a unique independent dominating set.

\begin{lemma}\label{arbs}
Let $T$ be either an arborescence or an anti-arborescence. Then $T$ admits a unique independent dominating set.
\end{lemma}
\begin{proof}
(Arborescences) The proof is by induction, and our basis is the trivial digraph. Let $T$ be an arborescence with root vertex $v$ with $|V(T)|=n>1$ and assume that $T$ is a minimum counterexample w.r.t $|V(T)|$. From the previous theorem we know that $T$ admits at least one independent dominating set. Consider the induced sub-tree (or possibly sub-forest) $T^{\prime}=T[V(T)\setminus N^{+}[v]]$. Since $T^{\prime}$ is smaller that $T$, $T^{\prime}$ admits a unique independent dominating set, call it $ID^{\prime}$ (technically, if $T$ is $K_{1,n-1}$ then $T^{\prime}$ and $ID^{\prime}$ are both empty, but this is not an issue). Since $v$ has in-degree zero, we know from Lemma \ref{indegzero} that $v$ must be in every independent dominating set of $T$ and also that $N^{+}(v)$ is not in any independent dominating set of $T$. This, along with the topology of $T$, implies that the set $ID=ID^{\prime}\cup\{v\}$ is the unique independent dominating set of $T$.

(Anti-arborescences) Assume that some anti-arborescence $T$ admits multiple independent dominating sets. Consider one independent dominating set of $T$, call it $ID_{1}$. Since $V(T)\setminus ID_{1}$ is independent ($T$ is a tree), it is also an independent dominating set in $T^{-}$, per Lemma \ref{reversereverse}. Next consider a distinct independent dominating set of $T$, call it $ID_{2}$. The set $V(T)\setminus ID_{2}$, which is not the same as the set $V(T)\setminus ID_{1}$, is also an independent dominating set of $T^{-}$. Since we now know that all arborescences admit a unique independent dominating set, we have derived a contradiction and conclude that all anti-arborescences must also admit a unique independent dominating set.
\end{proof}

\subsection{DAGs}
We next turn our attention to DAGs. Using the technique used to prove the uniqueness of independent dominating sets in arborescences, we can extend this result to DAGs simply by noticing that all DAGs have a sink vertex.

\begin{theorem}
Every DAG admits an independent dominating set.
\end{theorem}
\begin{proof}
Assume this is false. Let $D$ be a smallest possible DAG (w.r.t. $|V(D)|$) with no independent dominating set and let $v$ be a sink of $D$. Form the smaller DAG $D^{\prime}=D\setminus\{v\}$ (note that $D^{\prime}$, while certainly a DAG, might not be connected - this is OK). Since $D^{\prime}$ is smaller than $D$, it follows that $D^{\prime}$ admits some independent dominating set $ID$. If there exists some vertex $u\in ID\cap N^{-}(v)$ then we are done as $v$ is dominated by $ID$ in $D$. Otherwise the set $ID\cup\{v\}$ constitutes and independent dominating set of $D$. Thus no matter what we obtain an independent dominating set of $D$, contradicting $D$ being a minimum counterexample. Therefore no smallest counterexample exists and we conclude that every DAG admits an independent dominating set.
\end{proof}

\subsection{Cycles}
Turning our attention to cycles, we begin with what is perhaps an unexpected result. We have proven the existence of at least one independent dominating set for every member of every family of digraphs studied thus far. The first result in this section provides the first example of any digraph which does not admit a single independent dominating set. In fact, this example of a digraph with no independent dominating set is the only example found over the course of this paper.

\begin{lemma}\label{evencycles}
Let $C_{n}=v_{1}v_{2}\dots v_{n}v_{1}$ be a directed cycle. Then $C_{n}$ admits an independent dominating set if and only if $n\equiv 0\ (\mathrm{mod}\ 2)$. Additionally, $C_{n}$ admits exactly two distinct independent dominating sets when $n\equiv 0\ (\mathrm{mod}\ 2)$.
\end{lemma}
\begin{proof}
Let $n\equiv 0\ (\mathrm{mod}\ 2)$. Then each of $\{v_{i}\ |\ i\equiv 0\ (\mathrm{mod}\ 2)\}$ and $\{v_{i}\ |\ i\equiv 1\ (\mathrm{mod}\ 2)\}$ constitute independent dominating sets of $C_{n}$. Clearly $v_{i}$ and $v_{i+1}$ cannot both belong to an independent dominating set since $v_{i}v_{i+1}\in A(C_{n})$, so the only way a third independent dominating set could exist in $C_{n}$ is if neither $v_{i}$ nor $v_{i+1}$ are in some independent dominating set. But in this scenario the vertex $v_{i+1}$ is neither dominated by any vertex in the independent dominating set nor a member of the independent dominating set itself, hence the set is not actually a dominating set. Thus, when $n\equiv 0\ (\mathrm{mod}\ 2)$, there are exactly two distinct independent dominating sets of the directed cycle $C_{n}$.

Using the same argument, it follows that there are no independent dominating sets of the directed cycle $C_{n}$ when $n\equiv 1\ (\mathrm{mod}\ 2)$.
\end{proof}

\begin{corollary}
Not every digraph admits an independent dominating set.
\end{corollary}

In addition to proving that the directed odd cycle admits no independent dominating set, we have also shown that the directed even cycle admits exactly two distinct independent dominating sets. A far stronger positive existence result would be that all orientations of cycles, aside from the one counterexample presented above, admit independent dominating sets. This is exactly what we prove next.

\begin{theorem}\label{cyclesalmostall}
Let $C_{n}$ be any orientation of a cycle that is not a directed odd cycle. Then $C_{n}$ admits an independent dominating set.
\end{theorem}
\begin{proof}
Let $C_{n}=v_{1}v_{2}\dots v_{n}v_{1}$ be any orientation of any cycle that is not a directed odd cycle. We proceed by removing an arc from $C_{n}$, without loss of generality call it $v_{n}v_{1}$, and call the resulting path $P_{n}$. Since we have explicitly assumed that $C_{n}$ is not a directed odd cycle, and since we have fully characterized the independent dominating sets of the directed even cycle in Lemma \ref{evencycles}, we may assume that there exists at least one vertex with out-degree zero and one vertex of out-degree two in $C_{n}$.

If $C_{n}$ has the out-degree sequences $\{0,2,0,2,\dots,0,2,0,2\}$, then we may take the set of vertices of out-degree two as our independent dominating set, so let us assume that there exists at least one vertex of out-degree one in $C_{n}$. Without any loss of generality, we may further our assumption by asserting that $d^{+}(v_{n})=2$ and $d^{+}(v_{1})=1$.

Now, since $P_{n}$ is a path, it follow from Theorem \ref{pathstheorem} that $P_{n}$ admits an independent dominating set. Since both $v_{1}$ and $v_{n}$ have in-degree zero in $P_{n}$, we conclude from Lemma \ref{indegzero} that it must be the case that $\{v_{1},v_{n}\}\subseteq ID$ for all independent sets $ID$ of $P_{n}$. Furthermore, we may conclude that there exists some vertex $v_{i}\in V(P_{n})$ such that $d^{+}(v_{i})=0$.

Consider the path $P_{n-1}=C_{n}[V_{C_{n}\setminus\{v_{1}\}}]$. It still holds that $v_{n}$ is a member of every independent dominating set of $P_{n-1}$ since $v_{n}$ has in-degree zero in $P_{n-1}$. If $P_{n-1}$ admits an independent dominating set $ID$ such that $v_{2}\in ID$, then we are done since $ID$ constitutes and independent dominating set of $C_{n}$ since $v_{n}v_{1}\in A(C_{n})$. Thus, we have that there does not exist any independent dominating set of $P_{n-1}$ which contains the vertex $v_{2}$. This implies that $d^{+}_{P_{n-1}}(v_{2})=0$ since if $d^{+}_{P_{n-1}}(v_{2})=1$ would imply that $v_{2}$ would be in some independent dominating set of $P_{n-1}$. Since $d^{+}_{C_{n}}(v_{1})=1$, this in turn implies that $d^{+}_{C_{n}}(v_{2})=0$. This implies that $v_{3}$ must be a member of every dominating set of $P_{n-1}$. Fix some independent dominating set $ID$ of $P_{n-1}$. Since $v_{3}$ dominates $v_{2}$ and since $v_{n}v_{1}\in A(C_{n})$, it follows that $ID$ is an independent dominating set of $C_{n}$ and the proof is complete.
\end{proof}

\begin{corollary}
An orientation of a cycle admits an independent dominating set if and only if it is not a directed odd cycle.
\end{corollary}

\subsection{Bipartite Graphs}
The final family of digraphs we study in this paper is the family of oriented bipartite graphs. Given that bipartite graphs are comprised of two independent sets, they provide a natural place to begin looking for independent dominating sets. We denote bipartite graphs as $D=\{X,Y\}$ where $X$ and $Y$ are the two partite sets of $D$. Interestingly, it is quite easy to show that it is indeed possible for an independent dominating set of an oriented bipartite graph does not necessarily have to be comprised of a single partite set.  In general, this phenomenon is possible so long as there are no arcs from the portion of the independent set in $X$ to the portion of the independent dominating set in $Y$. To see this, consider the following example.

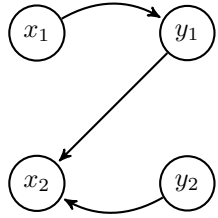
\begin{figure}[h!]
\centering
\label{bipic}
\begin{tikzpicture}[-,>=stealth',shorten >=1pt,auto,node distance=2cm,
                    thick,main node/.style={circle,draw}]
  \node[main node] (A)					      {$x_{1}$};
  \node[main node] (B) [below of=A]		      {$x_{2}$};
  \node[main node] (C) [right of=A] 	      {$y_{1}$};
  \node[main node] (D) [right of=B]      	  {$y_{2}$};
  \draw[thick,->]  (A) to [bend left]         (C);
  \draw[thick,->]  (C) to                     (B);
  \draw[thick,->]  (D) to [bend left]         (B);
\end{tikzpicture}
\caption{An example of an oriented bipartite graph $D=\{X,Y\}$ which admits an independent set comprised of members of both $X$ and $Y$. The set $ID=\{x_{1},y_{2}\}$ comprises such an independent dominating set.}
\end{figure}

The following lemma gives an immediate example of an oriented bipartite graph which does not admit an independent dominating set with vertices in each partite set.

\begin{lemma}
Let $K_{m,n}=\{X,Y\}$ be a directed complete bipartite graph, i.e., $\forall\ (x,y)\in X\times Y$, $xy\in A(K_{m,n})$. Then the set $X$ is the unique independent dominating set of $K_{m,n}$.
\end{lemma}
\begin{proof}
Since $xy\in A(K_{m,n})\ \forall (x,y)\in X\times Y$, it follows that $X$ dominates $Y$. Thus $X\cup N(X)=X\cup Y=V(K_{m,n})$ and so $X$ is a dominating set of $K_{m,n}$. Since $X$ is an independent set, not only is $X$ an independent dominating set of $K_{m,n}$, but $\forall x\in X\ \nexists\ v\in V(K_{m,n})$ such that $vx\in A(K_{m,n})$. Thus every member of $X$ must be in every independent dominating set of $K_{m,n}$. Since $X\cap ID\neq\emptyset$ for all independent dominating sets $ID$ of $K_{m,n}$, and since $xy\in A(K_{m,n})\ \forall y\in Y$ holds for each $x\in X$, it follows that $Y\cap ID=\emptyset$ for all independent dominating sets $ID$ of $K_{m,n}$. Therefore we conclude that $X$ is the unique independent dominating set of $K_{m,n}$.
\end{proof}

The conditions above, that the oriented bipartite graph be both complete and directed, can both be dropped with the partite $X$ remaining an independent dominating set.

\begin{lemma}
Let $D=\{X,Y\}$ be an orientation of a bipartite graph in which $X$ dominates $Y$, i.e., $D$ has the property that $\forall y\in Y$, $\exists x\in X$ such that $xy\in A(D)$. Then $X$ is an independent dominating set of $D$.
\end{lemma}
\begin{proof}
By definition of $D$ it follows that $X$ is a dominating set of $D$. Since $D=\{X,Y\}$ is bipartite, $X$ is independent.
\end{proof}

Note that in the more general case above, it is certainly possible that both partite sets may dominate the other, in which case both $X$ and $Y$ are independent dominating sets of $D$. This is quite interesting because the vertex set $V(D)$ is the union of two disjoint independent dominating sets. Characterizing the digraphs whose vertex sets are the union of $k$ disjoint independent dominating sets is an interesting problem which would significantly generalize this result. However, as we have yet to prove that all bipartite graphs admit an independent dominating set, we redirect our attention to this problem and prove a final lemma in preparation for the main existence theorem of this subsection.

\begin{lemma}\label{nozeros}
Let $D=\{X,Y\}$ be an orientation of a bipartite graph such that no vertex has in-degree equal to zero. Then $D$ admits an independent dominating set. Furthermore, any such independent dominating set is not unique.
\end{lemma}
\begin{proof}
Since every vertex has positive in-degree, it follows that $\forall y\in Y$, $\exists x\in X$ such that $xy\in A(D)$. Similarly, it follows that $\forall x\in X$, $\exists y\in Y$ such that $yx\in A(D)$. Thus we may choose either partite set to be our independent dominating set.
\end{proof}

Finally, we are ready to present an existence theorem which states that all bipartite graphs admit at least one independent dominating set.

\begin{theorem}\label{allbis}
Every oriented bipartite graph $D$ admits an independent dominating set.
\end{theorem}
\begin{proof}
The proof is by induction on $|V(D)|$, and the basis is trivial. Assume the claim is true for all oriented bipartite graphs on fewer than $n$ vertices, and let $D=\{X,Y\}$ be an orientation of a bipartite graph on $n$ vertices. From Lemma \ref{nozeros} we are done if there are no vertices with in-degree zero, so assume that there is at least one such vertex and call it $v$. Let $S=\{v\}\cup N^{+}(v)$ and create the induced sub-digraph $D^{\prime}=D[V(D)\setminus S]$. Since $|V(D)|<n$, it follows from our inductive hypothesis that $D^{\prime}$ admits an independent dominating set. Choose any independent dominating set of $D^{\prime}$, call it $ID$. Since $N^{+}(v)\not\in V(D^{\prime})$, and since $N^{+}(v)$ is dominated by $v$ in $D$, we may form the independent dominating set $ID^{\prime}=ID\cup\{v\}$ of $D$, and the proof is complete.
\end{proof}

\section{Idomatic Numbers}
The idomatic number of a (di)graph is the maximum number of pairwise disjoint independent dominating sets. Given the several uniqueness results proven in this paper, we can already state some results on the idomatic number of certain digraphs. Before providing specific results on the idomatic number of certain digraphs, we mention that for all families of digraphs studied in this paper, with the lone exception of the directed odd cycle, we know that the idomatic number $id(D)\geq 1$.

\begin{corollary}
If $T$ is a tournament then the idomatic number is given by
\begin{equation*}
id(T)=\begin{cases}
1\quad \mathrm{if}\ \exists\ v\in V(T)\ \mathrm{s.t.}\ vu\in A(T)\ \forall\ u\in V(T)\setminus\{v\}\\
0\quad \mathrm{otherwise}
\end{cases}
\end{equation*}
\end{corollary}
\begin{proof}
This follows directly from Theorem \ref{tourns}.
\end{proof}

\begin{corollary}
If $P$ is a directed path then the idomatic number is given by $id(P)=1$.
\end{corollary}
\begin{proof}
This follows directly from Lemma \ref{dipath}.
\end{proof}

\begin{corollary}
If $C_{n}$ is a directed cycle then the idomatic number is given by
\begin{equation*}
id(C_{n})=\begin{cases}
0\quad \mathrm{if}\ n\equiv 1\ (\mathrm{mod}\ 2) \\
2\quad \mathrm{if}\ n\equiv 0\ (\mathrm{mod}\ 2) 
\end{cases}
\end{equation*}
\end{corollary}
\begin{proof}
This follows directly from Lemma \ref{evencycles}.
\end{proof}

\begin{corollary}
If $T$ is either an arborescence or an anti-arborescence, then the idomatic number is given by $id(T)=1$.
\end{corollary}
\begin{proof}
This follows directly from Lemma \ref{arbs}.
\end{proof}

\section{Conclusion and Future Directions}
In this paper we began the study of independent dominating sets in digraphs. This is an extremely important first step for the development of routing mechanisms in directed networks, and can have profound applications in control systems engineering and other areas dependent upon directed communications. We achieved this by proving a series of existence and uniqueness results for important families of digraphs including orientations of complete graphs, paths, trees, DAGs, cycles, and bipartite graphs. We then studied the idomatic number which characterizes the number of disjoint independent dominating sets in a (di)graph and provided several results in the form of corollaries to the uniqueness results.

There are several possible directions in which one may further this line of research. Providing general bounds on the idomatic number in digraphs would be useful, as would characterizing those graphs with fixed idomatic numbers. In particular, characterizing all digraphs which have an idomatic number of 1 would be beneficial as this would bound those digraphs which might admit a unique independent dominating set. Furthermore, characterizing all digraphs which admit a unique independent dominating set would be an interesting problem, most notably because these graphs are a subset of the digraphs with idomatic number 1. 

Another interesting direction to explore would be to study the impact of a change in the orientation of a fixed underlying graph on the domatic number. While one may find the idomatic number for a particular orientation of a given graph, the idomatic number for the entire family, i.e., the minimum idomatic number over all possible orientations of a given graph, is a much more important problem, especially given the widespread application of independent dominating sets in ad hoc networks. A natural place to begin this inquiry could be to study the impact of reversals of orientation.

\bibliography{IDDG}

\begin{thebibliography}{10}

\bibitem{bau2002decycling}
Sheng Bau, Nicholas~C Wormald, and Sanming Zhou.
\newblock Decycling numbers of random regular graphs.
\newblock {\em Random Structures \& Algorithms}, 21(3-4):397--413, 2002.

\bibitem{cary2019trees}
Michael Cary.
\newblock Dominator chromatic numbers of orientations of trees.
\newblock {\em arXiv preprint arXiv:1904.06293}, 2019.

\bibitem{cary2019dominator}
Michael Cary.
\newblock Dominator colorings of digraphs.
\newblock {\em arXiv preprint arXiv:1902.07241}, 2019.

\bibitem{cockayne1995nordhaus}
Ernest~J Cockayne, Gerd Fricke, and Christina~M Mynhardt.
\newblock On a nordhaus-gaddum type problem for independent domination.
\newblock {\em Discrete mathematics}, 138(1-3):199--205, 1995.

\bibitem{cockayne1976disjoint}
Ernest~J Cockayne and Stephen~T Hedetniemi.
\newblock Disjoint independent dominating sets in graphs.
\newblock {\em Discrete Mathematics}, 15(3):213--222, 1976.

\bibitem{dai2006constructing}
Fei Dai and Jie Wu.
\newblock On constructing k-connected k-dominating set in wireless ad hoc and
  sensor networks.
\newblock {\em Journal of parallel and distributed computing}, 66(7):947--958,
  2006.

\bibitem{duckworth2002minimum}
William Duckworth and Nicholas~C Wormald.
\newblock Minimum independent dominating sets of random cubic graphs.
\newblock {\em Random Structures \& Algorithms}, 21(2):147--161, 2002.

\bibitem{duckworth2006independent}
William Duckworth and Nicholas~C Wormald.
\newblock On the independent domination number of random regular graphs.
\newblock {\em Combinatorics, Probability and Computing}, 15(4):513--522, 2006.

\bibitem{dunbar2006broadcasts}
Jean~E Dunbar, David~J Erwin, Teresa~W Haynes, Sandra~M Hedetniemi, and
  Stephen~T Hedetniemi.
\newblock Broadcasts in graphs.
\newblock {\em Discrete Applied Mathematics}, 154(1):59--75, 2006.

\bibitem{farber1984domination}
Martin Farber.
\newblock Domination, independent domination, and duality in strongly chordal
  graphs.
\newblock {\em Discrete Applied Mathematics}, 7(2):115--130, 1984.

\bibitem{favaron1990bound}
Odile Favaron.
\newblock {\em A bound on the independent domination number of a tree}.
\newblock Universit{\'e} Paris-Sud, Centre d'Orsay, Laboratoire de recherche en
  Informatique, 1990.

\bibitem{gera2007dominator}
Ralucca Gera.
\newblock On dominator colorings in graphs.
\newblock {\em Graph Theory Notes of New York}, 52:25--30, 2007.

\bibitem{glebov1998independent}
NI~Glebov and Alexandr~V Kostochka.
\newblock On the independent domination number of graphs with given minimum
  degree.
\newblock {\em Discrete mathematics}, 188(1-3):261--266, 1998.

\bibitem{goddard2013independent}
Wayne Goddard and Michael~A Henning.
\newblock Independent domination in graphs: A survey and recent results.
\newblock {\em Discrete Mathematics}, 313(7):839--854, 2013.

\bibitem{haynes2017domination}
Teresa~W Haynes.
\newblock {\em Domination in Graphs: Volume 2: Advanced Topics}.
\newblock Routledge, 2017.

\bibitem{haynes1998}
Teresa~W Haynes, Stephen Hedetniemi, and Peter Slater.
\newblock {\em Fundamentals of domination in graphs}.
\newblock CRC press, 1998.

\bibitem{irving1991approximating}
Robert~W Irving.
\newblock On approximating the minimum independent dominating set.
\newblock {\em Information Processing Letters}, 37(4):197--200, 1991.

\bibitem{ma2004note}
De-Xiang Ma and Xue-Gang Chen.
\newblock A note on connected bipartite graphs having independent domination
  number half their order.
\newblock {\em Applied mathematics letters}, 17(8):959--962, 2004.

\bibitem{prabhu2018independent}
S~Prabhu, T~Flora, and M~Arulperumjothi.
\newblock On independent resolving number of tio 2 [m, n] nanotubes.
\newblock {\em Journal of Intelligent \& Fuzzy Systems}, pages 1--5, 2018.

\bibitem{shao2019independent}
Zehui Shao, Zepeng Li, Aljo{\v{s}}a Peperko, Jiafu Wan, and Janez
  {\v{Z}}erovnik.
\newblock Independent rainbow domination of graphs.
\newblock {\em Bulletin of the Malaysian Mathematical Sciences Society},
  42(2):417--435, 2019.

\bibitem{sun1999upper}
Liang Sun and Jianfang Wang.
\newblock An upper bound for the independent domination number.
\newblock {\em Journal of Combinatorial Theory, Series B}, 76(2):240--246,
  1999.

\bibitem{valencia2010idomatic}
Mario Valencia-Pabon.
\newblock Idomatic partitions of direct products of complete graphs.
\newblock {\em Discrete Mathematics}, 310(5):1118--1122, 2010.

\bibitem{wu2006extended}
Jie Wu, Mihaela Cardei, Fei Dai, and Shuhui Yang.
\newblock Extended dominating set and its applications in ad hoc networks using
  cooperative communication.
\newblock {\em IEEE Transactions on Parallel and Distributed Systems},
  17(8):851--864, 2006.

\bibitem{wu1999calculating}
Jie Wu and Hailan Li.
\newblock On calculating connected dominating set for efficient routing in ad
  hoc wireless networks.
\newblock In {\em Proceedings of the 3rd international workshop on Discrete
  algorithms and methods for mobile computing and communications}, pages 7--14.
  Citeseer, 1999.

\bibitem{zelinka1983adomatic}
Bohdan Zelinka.
\newblock Adomatic and idomatic numbers of graphs.
\newblock {\em Mathematica Slovaca}, 33(1):99--103, 1983.

\end{thebibliography}
\end{document}